\documentclass[11pt,reqno]{amsart}

\usepackage{amsfonts, amsthm, amsmath}
\allowdisplaybreaks[4]

\usepackage{rotating}

\usepackage{tikz}

\usepackage{graphics}

\usepackage{amssymb}

\usepackage{circledsteps}

\usepackage{amsmath}

\usepackage{amscd}

\usepackage[latin2]{inputenc}

\usepackage{t1enc}

\usepackage[mathscr]{eucal}

\usepackage{indentfirst}

\usepackage{graphicx}

\usepackage{graphics}

\usepackage{pict2e}

\usepackage{mathrsfs}

\usepackage{enumerate}

\usepackage[pagebackref]{hyperref}
\hypersetup{colorlinks=true}
\usepackage{cite}
\usepackage{color}
\usepackage{epic}
\usepackage{hyperref} 
\usepackage{framed}
\usepackage{mathabx}

\numberwithin{equation}{section}
\theoremstyle{plain}

\newtheorem{theorem}{Theorem}[section]

\newtheorem{lemma}[theorem]{Lemma}

\theoremstyle{definition}

\newtheorem{Def}[theorem]{Definition}

\newtheorem{example}[theorem]{Example}

\newtheorem{remark}[theorem]{Remark}

\newtheorem{?}[theorem]{Problem}

\newtheoremstyle{named}{}{}{\itshape}{}{\bfseries}{.}{.5em}{#1\thmnote{ #3}}
\theoremstyle{named}

\newcommand{\f}[1]{\ifthenelse{\equal{#1}{1}}{(q;q)_\infty}{(q^{#1};q^{#1})_{\infty}}}

\def\ped{\mathrm{ped}}
\def\sP{\mathscr{P}}
\def\sF{\mathscr{F}}
\def\sV{\mathscr{V}}
\def\sD{\mathscr{D}}
\def\sA{\mathscr{A}}
\def\sB{\mathscr{B}}
\def\sC{\mathscr{C}}

\def\ri{\rightarrow}
\def\la{\lambda}
\def\al{\alpha}
\def\ep{\epsilon}
\def\si{\sigma}
\def\xri{\xrightarrow}

\begin{document}
\title[Combinatorial perspectives on identities for partitions with distinct even parts]{Combinatorial perspectives on identities for partitions with distinct even parts}

\author[H. Li]{Haijun Li}
\address[Haijun Li]{College of Mathematics and Statistics, Chongqing University, Chongqing 401331, P.R. China}
\email{lihaijun@cqu.edu.cn; lihaijune@163.com}

\date{\today}

\begin{abstract}
Partitions with distinct even parts have long been the subject of extensive research. In this paper, We present some new perspectives on such partitions from a combinatorial viewpoint, and connect them with signed partitions and bicolored partitions, thereby obtaining several partition identities. We construct bijective proofs for each of our results. Furthermore, these bijections will partially answer the combinatorial problems posed by Andrews-El Bachraoui and K\i l\i \c c-Kur\c sung\"oz. respectively. 
\end{abstract}


\maketitle
\section{Introduction}\label{sec:intro}

A {\it partition} $\la$ of a positive integer $n$ is a nonincreasing sequence of positive integer $(\la_1, \la_2, ..., \la_\ell)$ such that $\la_1+\la_2+\cdots +\la_\ell=n$ (see \cite{andtp}). Each $\la_i$ is called a {\it part} of the partition $\la$, and $n$ is called the {\it weight} of $\la$ and denote $|\la|=n$. The {\it length} $\ell=\ell(\la)$ of $\la$ is the number of parts in $\la$. As a convention, we denote the partition of $0$ by $\ep$. For any two partitions $\la$ and $\mu$, $\la\oplus\mu$ gives a partition by collecting all parts in $\la$ and $\mu$.

In 2009, Andrews \cite{and09} denoted the $\ped(n)$ as the number of partitions of $n$ with distinct even parts (odd parts are unrestricted) and found that 
\begin{align}\label{id:gf_ped}
\sum_{n\geq 0}\ped(n)q^n=\frac{(-q^2; q^2)_{\infty}}{(q; q^2)_{\infty}}=1+q+2q^2+3q^3+4q^4+6q^5+9q^6+12q^7+\cdots,
\end{align}
where we employ the standard notation~\cite{GR90}
\begin{align*}
(a; q)_0:=1,\ (a; q)_n:=\prod_{k=0}^{n-1}(1-aq^{k}),\ n\geq 1,\text{ and }(a; q)_{\infty}:=\lim_{n\ri \infty}(a; q)_n,\ |q|<1.
\end{align*}
This partition, along with its arithmetic properties, has been extensively studied in recent years, see for instance \cite{AB251, AHS10, che11, mer17}. Recently, Andrews and El Bachraoui~\cite{AB252} introduced a new set of partitions equinumerous with those with distinct even parts, and asked for a combinatorial proof of the following theorem.

\begin{theorem}[{cf. \cite[Definition 1 and Corollary 1]{AB252}}]\label{thm:AB_Fn}
For any nonnegative integer $n$, let $F(n)$ denote the number of partitions of $n$ where the smallest part is $1$, the first occurrence of $1$ may be overlined, each part is at most twice the number of the occurrences of $1$, and the remaining odd parts are not repeated. Then we have
\begin{align*}
F(n)=\ped(n).
\end{align*}
\end{theorem}

Let $\sP_{ed}(n)$ and $\sF(n)$ denote the sets of partitions counted by $\ped(n)$ and $F(n)$, respectively. Our first objective is to relate $\sP_{ed}(n)$ and $\sF(n)$ to a class of {\it signed partitions}, leading to a refined three-parameter identity among the three partition sets. Signed partitions were first introduced by Andrews \cite{and07} and further developed recently by Alanazi, Munagi and Sills~\cite{AMS25}. We now proceed to present the formal definition of signed partitions.

\begin{Def}
A {\it signed partition} $\si$ of $n$ is a pair $(\pi, \nu)$ of partitions such that $|\si|=|\pi|-|\nu|=n$. The members of $\pi$ (resp. $\nu$) are the {\it positive parts} (resp. {\it negative parts}) of the signed partition $\si$. Define $\ell^{+}(\si)=\ell(\pi)$ (resp. $\ell^{-}(\si)=\ell(\nu)$) to be the number of positive (resp. negative) parts in $\si$. For example, $\si=((4, 4, 2, 1), (2, 1, 1))$ is a signed partition of $11-4=7$ and $\ell^+(\si)=4$ and $\ell^-(\si)=3$.
\end{Def}

We require some further notation. Let $\ell_e(\la)$ denote the number of even parts in $\la\in \sP_{ed}(n)$. Let $\ell_1(\la)$ and $\overline{\ell}_o(\la)$ denote, respectively, the number of parts equal to $1$, and the number of odd parts other than $1$ together with $\overline{1}$ (if it exists) in $\la\in \sF(n)$. With these preliminaries in place, we may state our first main result.

\begin{theorem}\label{thm:res1}
Let $n$, $m$ and $k$ be nonnegative integers. Let $\sF_{-1}(n)$ denote the set of signed partitions $\si$ of $n$ where positive parts are even and distinct, and negative parts are odd, distinct and at most $2\ell^+(\si)-1$. Further, we define three partition sets as follows.
\begin{align*}
\begin{array}{c|c}
\hline\\[-6pt]
\text{Set of Partitions} & \text{Corresponding Cardinality}\\[6pt]
\hline\\[-6pt]
\sP_{ed}(n, m, k):=\{\la\in \sP_{ed}(n): \ell(\la)=m, \ell_e(\la)=k\} & \ped(n, m, k)=|\sP_{ed}(n, m, k)|\\[6pt]
\hline\\[-6pt]
\sF(n, m, k):=\{\la\in \sF(n): \ell_1(\la)=m, \overline{\ell}_o(\la)=k\}& F(n, m, k)=|\sF(n, m, k)|\\[6pt]
\hline\\[-6pt]
\sF_{-1}(n, m, k):=\{\si\in \sF_{-1}(n): \ell^+(\si)=m,\ell^+(\si)-\ell^-(\si)=k\}& F_{-1}(n, m, k)=|\sF_{-1}(n, m, k)|\\[6pt]
\hline
\end{array}
\end{align*}
Then we have
\begin{align}
\ped(n, m, k)=F(n, m, k)=F_{-1}(n, m, k).
\end{align}
\end{theorem}

Another concept closely related to partitions with distinct even parts is the Lebesgue identity:
\begin{align}\label{id:lebesgue}
\sum_{n\geq 0}\frac{(-xq; q)_n}{(q; q)_n}q^{\binom{n+1}{2}}=(-xq^2; q^2)_{\infty}(-q; q)_{\infty}.
\end{align}
Let $\sD$ be the set of partitions into distinct parts. On the side of the series, one can consider the partition set $\sV(n, k)$ consisting of partition pairs $(\al, \beta)$ where $\al\in \sD$ and $\beta\in \sD$ satisfying that $|\al|+|\beta|=n$, $\ell(\beta)=k$ and $\max(\beta)\leq \ell(\al)$. Note that $\max(\la)$ is the largest part of a partition $\la$. On the side of the product, using Euler's theorem (that is, $1/(q; q^2)_{\infty}=(-q; q)_{\infty}$), this is the generating function for partitions $\la\in \sP_{ed}(n, k)=\bigcup_{m\geq 0}\sP_{ed}(n, m, k)$. Bessenrodt~\cite{bes94} gave a bijection between $\sV(n, k)$ and $\sP_{ed}(n, k)$ in term of $2$-modular diagrams. For more combinatorial proofs of the Lebesgue identity and its extensions, see \cite{AG93, CHS10, fu08, LS09, RV72, row10}. Our second main result is to present two new combinatorial perspectives on the series side of the Lebesgue identity, one of which is related to signed partitions. We state the result as follows.

\begin{theorem}\label{thm:res2}
Let $n$ and $k$ be nonnegative integers. Let $\sA(n)$ denote the set of weighted partitions $\la\in \sD$ of $n$ where if $\la_i$ is labeled as $x$ then $\la_i-\la_{i+1}\geq 2$ (we adopt the convention that $\la_{n+1}=0$ so that $1$ will not be labeled as $x$). Let $\sA_{-1}(n)$ denote the set of signed partitions $\si$ of $n$ where positive parts differ by at least $2$ and the smallest positive part is at least $2$, and negative parts are distinct and at most $\ell^{+}(\si)$. Further, letting $\ell_x(\la)$ be the number of parts labeled as $x$ in $\la\in \sA(n)$, we define three partition sets as follows.
\begin{align*}
\begin{array}{c|c}
\hline\\[-6pt]
\text{Set of Partitions} & \text{Corresponding Cardinality}\\[6pt]
\hline\\[-6pt]
\sV(n, k)\text{ is as defined above} & V(n, k)=|\sV(n, k)|\\[6pt]
\hline\\[-6pt]
\sA(n, k):=\{\la\in \sA(n): \ell_x(\la)=k\}& A(n, k)=|\sA(n, k)|\\[6pt]
\hline\\[-6pt]
\sA_{-1}(n, k):=\{\si\in \sA_{-1}(n): \ell^+(\si)-\ell^-(\si)=k\}& A_{-1}(n, k)=|\sA_{-1}(n, k)|\\[6pt]
\hline
\end{array}
\end{align*}
Then we have
\begin{align}
V(n, k)=A(n, k)=A_{-1}(n, k).
\end{align}
\end{theorem}

Recently, K\i l\i \c c and Kur\c sung\"oz~\cite{KK25} developed a search algorithm for systems of $q$-difference equations satisfied by Andrews-Gordon type double series and explained some of the double series in a base partition and moves framework. They considered the set $\sB(n)$ of bicolored partitions of $n$ such that
\begin{enumerate}[(a)]
\item Neither blue nor red parts repeat.

\item When we go through parts from the smallest to the largest, for each red $c_r$, there must be a blue $c_b$ or a blue $(c+1)_b$. Moreover, these blue parts should be different for each red $c_r$.
\end{enumerate}
For example, the four bicolored partitions in $\sB(4)$ are 
\begin{align*}
4_b,\ 3_b+1_b,\ 2_b+2_r,\ 2_b+1_b+1_r.
\end{align*}
Let $\ell_r(\la)$ be the number of red parts in $\la\in \sB(n)$. Therefore, we may state our final result as follows, which partially answers the combinatorial problem K\i l\i \c c and Kur\c sung\"oz posed at the end of their paper.

\begin{theorem}\label{thm:res3}
Let $n$ and $k$ be nonnegative integers and 
\begin{align*}
\sB(n, k):=\{\la\in \sB(n): \ell_r(\la)=k\},\quad B(n, k)=|\sB(n, k)|.
\end{align*}
Then we have
\begin{align*}
\ped(n, k)=B(n, k).
\end{align*}
\end{theorem}

We organize the rest of the paper as follows. In Sections \ref{sec:res1}, \ref{sec:res2} and \ref{sec:res3}, we present proofs (that are analytic, combinatorial, or both) for Theorems \ref{thm:res1}, \ref{thm:res2} and \ref{thm:res3}, respectively. The bijections that we will construct in Sections \ref{sec:res1} and \ref{sec:res3} partially answer the combinatorial problems posed by Andrews-El Bachraoui and K\i l\i \c c-Kur\c sung\"oz, respectively. 

\section{Proof of Theorem \ref{thm:res1}}\label{sec:res1}

In this section, we shall present both an analytic proof and a combinatorial proof of Theorem \ref{thm:res1}. We begin with the analytic proof, for which we require a prerequisite $q$-series identity, namely the $q$-binomial theorem~\cite[Eq. (II.3)]{GR90}:
\begin{align}\label{id:qbino}
\sum_{n\geq 0}\frac{(a; q)_nz^n}{(q; q)_n}=\frac{(az; q)_{\infty}}{(z; q)_{\infty}}.
\end{align}

\begin{proof}[Analytic proof of Theorem \ref{thm:res1}]
Let $q\ri q^2$, $a\ri -yq$ and $z\ri xq$ in equation \eqref{id:qbino}, then
\begin{align}\label{id:qbino-xy}
\sum_{n\geq 0}\frac{(-yq; q^2)_nx^nq^n}{(q^2; q^2)_{n}}=\frac{(-xyq^2; q^2)_{\infty}}{(xq; q^2)_{\infty}}.
\end{align}
Now we have
\begin{align*}
\sum_{n, m, k\geq 0}\ped(n, m, k)x^my^kq^n&=\sum_{n, m, k\geq 0}\sum_{\la\in \sP_{ed}(n, m, k)}x^{\ell(\la)}y^{\ell_e(\la)}q^{|\la|}\\
&=\frac{(-xyq^2; q^2)_{\infty}}{(xq; q^2)_{\infty}}\\
&=\sum_{n\geq 0}\frac{(-yq; q^2)_nx^nq^n}{(q^2; q^2)_{n}}\quad\quad \text{(using equation \eqref{id:qbino-xy})}\\
&=\sum_{n\geq 0}\frac{(1+yq)(1+yq^3)\cdots (1+yq^{2n-1})}{(1-q^2)(1-q^4)\cdots (1-q^{2n})}(xq)^{\overbrace{1+1+\cdots +1}^{\text{n times}}}\\
&=\sum_{n, m, k\geq 0}\sum_{\la\in \sF(n, m, k)}x^{\ell_1(\la)}y^{\overline{\ell}_o(\la)}q^{|\la|}\\
&=\sum_{n, m, k\geq 0}F(n, m, k)x^my^kq^n.
\end{align*}
In fact, this also explains why the first occurrence of $1$ in a partition $\la\in \sF(n, m, k)$ may be overlined. Since $\overline{1}$ arises from $(1+yq)$ at most once and is tracked by $y$, it is distinguished from the $1$'s coming from $(xq)^n$ which are tracked by $x$.

On the other hand, we have
\begin{align*}
\sum_{n, m, k\geq 0}F(n, m, k)x^my^kq^n&=\sum_{n\geq 0}\frac{(1+yq)(1+yq^3)\cdots (1+yq^{2n-1})}{(1-q^2)(1-q^4)\cdots (1-q^{2n})}(xq)^{n}\\
&=\sum_{n\geq 0}\frac{(\frac{1}{yq}+1)(\frac{1}{yq^3}+1)\cdots (\frac{1}{yq^{2n-1}}+1)}{(1-q^2)(1-q^4)\cdots (1-q^{2n})}x^{n}y^nq^{n^2+n}\\
&=\sum_{n\geq 0}\frac{(xy)^nq^{2+4+\cdots +(2n)}}{(1-q^2)\cdots (1-q^{2n})}\times \left[(\frac{1}{yq}+1)(\frac{1}{yq^3}+1)\cdots (\frac{1}{yq^{2n-1}}+1)\right]\\
&=\sum_{n, m, k\geq 0}\sum_{\si=(\pi, \nu)\in \sF_{-1}(n, m, k)}\left[(xy)^{\ell^+(\si)}q^{|\pi|}\right]\times \left[\frac{1}{y^{\ell^-(\si)}q^{|\nu|}}\right]\\
&=\sum_{n, m, k\geq 0}\sum_{\si\in \sF_{-1}(n, m, k)}x^{\ell^+(\si)}y^{\ell^+(\si)-\ell^-(\si)}q^{|\si|}\\
&=\sum_{n, m, k\geq 0}F_{-1}(n, m, k)x^my^kq^n.
\end{align*}
Thus, the proof is complete.
\end{proof}

The next part of this section is to provide a combinatorial proof of Theorem \ref{thm:res1}, which we will accomplish by constructing two bijections. The first bijection is between $\sP_{ed}(n, m, k)$ and $\sF(n, m, k)$, which we present as a lemma below.

\begin{lemma}\label{lem:f1}
For any nonnegative integers $n$, $m$ and $k$, there exists a bijection
\begin{align*}
f_1: \sP_{ed}(n, m, k)& \ri \sF(n, m, k)\\
\la & \mapsto \mu
\end{align*}
such that $|\la|=|\mu|$, $\ell(\la)=\ell_1(\mu)$ and $\ell_e(\la)=\overline{\ell}_o(\mu)$.
\end{lemma}

\begin{proof}
First, we describe the map $f_1$ from $\sP_{ed}(n, m, k)$ to $\sF(n, m, k)$. Given a partition $\la=\la_1+\la_2+\cdots +\la_m\in \sP_{ed}(n, m, k)$, we divide this map into the following three steps.
\begin{description}
\item[Step 1] Since $\ell_e(\la)=k$, let $1\leq i_1<i_2<\cdots <i_k\leq m$ be the positions of all even parts in $\la$. Then we can perform the following operation on each even part together with all the parts preceding it.
\begin{align*}
\la\ri \la'=\la_1'+\cdots +\la_m',\text{ where }\la_j'=\begin{cases}
\la_j-2k+2t & \text{ if }i_t<j<i_{t+1}\text{ for }0\leq t\leq k,\\
\la_j-2k+2t-1 & \text{ if }j=i_t\text{ for }1\leq t\leq k,
\end{cases}
\end{align*}
and we set $i_0=0$ and $i_{k+1}=m+1$. Note that $\la'$ is well-defined and is a partition into $m$ odd parts. From the parts that have been cut off, we obtain a new partition $$\pi=(2i_k-1)+(2i_{k-1}-1)+\cdots +(2i_1-1)$$ with $k$ distinct odd parts and the largest part not exceeding $2m-1$. We can readily verify that
\begin{align*}
|\la'|+|\pi|&=|\la|+\sum_{t=0}^{k-1}(-2k+2t)(i_{t+1}-i_t-1)+\sum_{t=1}^k(-2k+2t-1)+\sum_{t=1}^k(2i_t-1)\\
&=|\la|+\sum_{t=1}^k(-2i_t)+k(k+1)+(-2k^2+k^2)+\sum_{t=1}^k(2i_t)-k\\
&=|\la|.
\end{align*}

\item[Step 2] In this step, we will separate a $1$ from each part of $\la'$. Without loss of generality, assume that there exists some $1\leq t\leq m$ such that all parts after $\la_t'$ in $\la'$ are equal to $1$, i.e.,
\begin{align*}
\la_1'\geq\la_2'\geq \cdots \geq \la_t'>\la_{t+1}'=\cdots =\la_m'=1.
\end{align*}
Then we can perform the following operations:
\begin{align*}
\la'\ri \la''\oplus (\underbrace{1+\cdots +1}_{\text{m times}})=(\la_1''+\cdots +\la_t'')\oplus (\underbrace{1+\cdots +1}_{\text{m times}})
\end{align*}
where $\la_j''=\la_j'-1$ for all $1\leq j\leq t$. Note that $\la''$ is well-defined and is a partition into $t$ even parts.

\item[Step 3] We perform the following decomposition and reconstruction on $\la''$:
\begin{align*}
\la''\ri \la'''=\bigoplus_{s=1}^t(\underbrace{(2s)+\cdots +(2s)}_{\text{$(\la_s''-\la_{s+1}'')/2$ times}})
\end{align*}
where we set $\la_{t+1}''=0$. Note that $|\la'''|=\sum_{s=1}^t2s(\la_s''-\la_{s+1}'')/2=|\la''|$.
\end{description}
Finally, we obtain $$\mu=\la'''\oplus\pi\oplus(\underbrace{1+\cdots +1}_{\text{m times}})$$ where if $2i_1-1=1$ in $\pi$ then we need to overline it. Through the construction process described above, we can clearly see that $\mu\in \sF(n, m, k)$. On the other hand, the inverse map $f_1^{-1}$ can be obtained by reversing the above construction step by step, thus completing the proof.
\end{proof}



Next, we present the second bijection, which is between $\sP_{ed}(n, m, k)$ and $\sF_{-1}(n, m, k)$.

\begin{lemma}\label{lem:f2}
For any nonnegative integers $n$, $m$ and $k$, there exists a bijection
\begin{align*}
f_2: \sP_{ed}(n, m, k)&\ri \sF_{-1}(n, m, k)\\
\la & \mapsto \si
\end{align*}
such that $|\la|=|\si|$, $\ell(\la)=\ell^+(\si)$ and $\ell_e(\la)=\ell^+(\si)-\ell^-(\si)$.
\end{lemma}

\begin{proof}
For a given partition $\la\in \sP_{ed}(n, m, k)$, we know that it has $m-k$ odd parts. Then, without loss of generality, assume that $1\leq i_1< i_2< \cdots < i_{m-k}\leq m$ are the positions of all its odd parts. We can describe the operation of this map $f_2$ as follows:
\begin{align*}
\la\ri \pi=\pi_1+\cdots +\pi_m,\text{ where }\pi_j=\begin{cases}
\la_j+2(m-k-t) & \text{ if }i_t<j<i_{t+1}\text{ for }0\leq t\leq m-k,\\
\la_j+2(m-k-t)+1 & \text{ if }j=i_t\text{ for }1\leq t\leq m-k,
\end{cases}
\end{align*}
and we set $i_0=0$ and $i_{m-k+1}=m+1$. Moreover, we have
\begin{align*}
\nu=(2i_{m-k}-1)+(2i_{m-k-1}-1)+\cdots +(2i_1-1).
\end{align*} 
Note that $\nu$ is a partition into $m-k$ distinct odd parts and $\max(\nu)=2i_{m-k}-1\leq 2m-1$. Therefore, we obtain $\si=(\pi, \nu)\in \sF_{-1}(n, m, k)$. One can readily verify that 
\begin{align*}
|\si|&=|\pi|-|\nu|\\
&=|\la|+\sum_{t=0}^{m-k}2(m-k-t)(i_{t+1}-i_{t}-1)+\sum_{t=1}^{m-k}2(m-k-t)+(m-k)-\sum_{t=1}^{m-k}(2i_t)+(m-k)\\
&=|\la|+\sum_{t=1}^{m-k}(2i_{t})-(m-k+1)(m-k)+(m-k)(m-k-1)+(m-k)-\sum_{t=1}^{m-k}(2i_t)+(m-k)\\
&=|\la|.
\end{align*}
On the other hand, for the inverse map $f_2^{-1}$, we can reverse the above construction step by step, which will not be elaborated here. Thus, the proof is complete.
\end{proof}


\section{Proof of Theorem \ref{thm:res2}}\label{sec:res2}

Before presenting the proof of Theorem \ref{thm:res2}, let us first examine how the partition set $\sA(n, k)$ is derived from the series part of the Lebesgue identity. Recall the series and expand it:
\begin{align*}
\sum_{n\geq 0}\frac{(-xq; q)_n}{(q; q)_n}q^{\binom{n+1}{2}}=\sum_{n\geq 0}\frac{q^{1+2+\cdots +n}}{(q; q)_n}\times\left[(1+xq)(1+xq^2)\cdots (1+xq^n)\right].
\end{align*}
For the $n$-th summand in the summation, the factor $q^{\binom{n+1}{2}}/(q; q)_n$ generates all partitions with precisely $n$ distinct parts. The remaining factors create weights for each part. Namely, for $1\leq j\leq n$, the $j$-th largest part receives either a weight $1$ or $x$, and in the latter case all the previous $j$ parts get increased by $1$ in size to accommodate the extra weight of $q^j$. Consequently, we see that the summation generates the set of partitions $\sA(n, k)$. Based on this combinatorial interpretation, we can easily obtain the following bijection:

\begin{lemma}\label{lem:g1}
For any nonnegative integers $n$ and $k$, there exists a bijection
\begin{align*}
g_1:\sV(n, k)& \ri \sA(n, k)\\
(\al, \beta)&\mapsto\la
\end{align*}
such that $|\al|+|\beta|=|\la|$ and $\ell(\beta)=\ell_x(\la)$.
\end{lemma}

In view of these preliminary discussions, we only need to provide the proof for $A(n, k)=A_{-1}(n, k)$. First, let us examine the analytic proof.

\begin{proof}[Analytic of Theorem \ref{thm:res2}]
We have
\begin{align*}
\sum_{n, k\geq 0}A(n, k)x^kq^n&=\sum_{n, k\geq 0}\sum_{\la\in \sA(n, k)}x^{\ell_x(\la)}q^{|\la|}\\
&=\sum_{n\geq 0}\frac{(1+xq)(1+xq^2)\cdots (1+xq^n)}{(1-q)(1-q^2)\cdots (1-q^{n})}q^{\frac{n^2+n}{2}}\\
&=\sum_{n\geq 0}\frac{(\frac{1}{xq}+1)(\frac{1}{xq^2}+1)\cdots (\frac{1}{xq^n}+1)}{(1-q)(1-q^2)\cdots (1-q^{n})}x^nq^{n^2+n}\\
&=\sum_{n\geq 0}\frac{x^nq^{2+4+\cdots +2n}}{(1-q)(1-q^2)\cdots (1-q^{n})}\times\left[(\frac{1}{xq}+1)(\frac{1}{xq^2}+1)\cdots (\frac{1}{xq^n}+1)\right]\\
&=\sum_{n, k\geq 0}\sum_{\si=(\pi, \nu)\in\sA_{-1}(n, k)}\left[x^{\ell^+(\si)}q^{|\pi|}\right]\times \left[\frac{1}{x^{\ell^-(\si)}q^{|\nu|}}\right]\\
&=\sum_{n, k\geq 0}\sum_{\si\in \sA_{-1}(n, k)}x^{\ell^+(\si)-\ell^-(\si)}q^{|\si|}\\
&=\sum_{n, k\geq 0}A_{-1}(n, k)x^ky^n.
\end{align*}
\end{proof}

We conclude this section with the construction of the following bijection.

\begin{lemma}\label{lem:g2}
For any nonnegative integers $n$ and $k$, there exists a bijection 
\begin{align*}
g_2: \sA(n, k)&\ri \sA_{-1}(n, k)\\
\la&\mapsto \si
\end{align*}
such that $|\la|=|\si|$ and $\ell_x(\la)=\ell^+(\si)-\ell^-(\si)$.
\end{lemma}

\begin{proof}
For a given partition $\la=\la_1+\la_2+\cdots +\la_m\in \sA(n, k)$, without loss of generality, assume that $1\leq i_1<i_2<\cdots <i_k\leq m$ are the positions of all parts not labeled as $x$. We can describe the operation of this map $g_2$ as follows:
\begin{align*}
\la\ri \pi=\pi_1+\cdots +\pi_m,\text{ where }\pi_j=\la_j+k-t\text{ if }i_t<j\leq i_{t+1} \text{ for }0\leq t\leq k,
\end{align*}
and we set $i_0=0$ and $i_{k+1}=m$. Note that if we remove all labels in $\pi$, then $\pi$ is a partition where parts differ by at least $2$ and the smallest part is at least $2$. Moreover, we have
\begin{align*}
\nu=i_k+i_{k-1}+\cdots +i_{1},
\end{align*}
 and it is a partition where parts are distinct and at most $i_k\leq \ell(\pi)=m$. Therefore, we obtain $\si=(\pi, \nu)\in \sA_{-1}(n, k)$. One can readily check that 
\begin{align*}
|\si|&=|\pi|-|\nu|\\
&=|\la|+\sum_{t=0}^k(k-t)(i_{t+1}-i_{t})-\sum_{t=1}^ki_t\\
&=|\la|.
\end{align*}
On the other hand, for the inverse map $g_2^{-1}$, we can reverse the above construction step by step, which will not be elaborated here. Thus, the proof is complete.
\end{proof}

\begin{remark}
From the bijection $g_2$ constructed in Lemma \ref{lem:g2}, we can see that it in fact preserves the number of parts, i.e., $\ell(\la)=\ell^+(g_2(\la))$ for any $\la\in \sA(n, k)$.
\end{remark}

\section{Proof of Theorem \ref{thm:res3}}\label{sec:res3}

In this section, we focus our attention on the combinatorial proof of Theorem \ref{thm:res3}. Before proceeding, we need to make a slight modification using Euler's theorem. One has
\begin{align*}
\sum_{n, k\geq 0}\sP_{ed}(n, k)x^kq^n=\frac{(-xq^2; q^2)_{\infty}}{(q; q^2)_{\infty}}=(-q; q)_{\infty}(-xq^2; q^2)_{\infty}=\sum_{n, k\geq 0}\sC(n, k)x^kq^n,
\end{align*}
where $\sC(n, k)$ is the set of partition pairs $(\al, \beta)$ of $n$ where $\al$ is a partition into distinct parts and $\beta$ is a partition into distinct even parts with $\ell(\beta)=k$. 

Next, and most importantly, we shall establish a bijection between $\sC(n, k)$ and $\sB(n, k)$ as follows, which will serve to prove Theorem \ref{thm:res3}. 

\begin{lemma}\label{lem:h}
For any nonnegative integers $n$ and $k$, there exists a bijection 
\begin{align*}
h: \sC(n, k)&\ri \sB(n, k)\\
(\al, \beta)&\mapsto \la
\end{align*}
such that $|\al|+|\beta|=|\la|$ and $\ell(\beta)=\ell_r(\la)$.
\end{lemma}

\begin{proof}
For a given partition pair $(\al, \beta)\in \sC(n, k)$, we first describe the map $h$ from $\sC(n, k)$ to $\sB(n, k)$. Note that when $k=0$, that is, $\beta=\ep$ (the empty partition), we only need to color all parts of $\al$ blue to obtain $\la\in\sB(n, 0)$. Henceforth, for the case $k\geq 1$, we divide the description of this map into the following three steps.
\begin{description}
\item[Step 1] Assume that $\al_{p+1}$ is the largest part of $\al$ not exceeding $\ell(\beta)$. Then we can perform the following operations:
\begin{align*}
\al\ri\al'=\al_1+\al_2+\cdots +\al_{p}.
\end{align*}
Suppose that $\al_{t}=s_{\ell(\al)+1-t}$ for all $p+1\leq t\leq \ell(\al)$, then we have
\begin{align*}
\beta\ri \beta'=\beta_1'+\beta_2'+\cdots +\beta_k',\end{align*}
where $$\beta_j'=\beta_j+\ell(\al)-p-t\text{ if }s_{t}+1\leq j\leq s_{t+1}\text{ for }0\leq t\leq \ell(\al)-p,$$
and we set $s_0=0$ and $s_{\ell(\al)-p+1}=k$. One can readily verify that 
\begin{align*}
|\beta'|=|\beta|+\sum_{t=0}^{\ell(\al)-p}(s_{t+1}-s_t)(\ell(\al)-p-t)=|\beta|+\sum_{t=1}^{\ell(\al)-p}s_t=|\beta|+\sum_{t=p+1}^{\ell(\al)}\al_t.
\end{align*}

\item[Step 2] We decompose each $\beta_j'$ in $\beta'$ into a pair as follows:
\begin{align*}
\beta_j'\ri \beta_j''=\begin{cases}
\dfrac{\beta_j'}{2}+(\dfrac{\beta_j'}{2})_r & \text{ if }\beta_j'\text{ even},\\
\dfrac{\beta_j'+1}{2}+(\dfrac{\beta_j'-1}{2})_r & \text{ if }\beta_j'\text{ odd},
\end{cases}
\end{align*}
while keeping their original position in $\beta'$. This yields a new colored partition $\beta''=\beta_1''+\cdots +\beta_k''$, where each $\beta_j''$ actually consists of two parts; hence $\ell(\beta'')=2k$.

\item[Step 3] We place each $\beta_j''$ at the end of $\al'$, and then perform the following weight-preserving ``forward adjustment'' until this pair of parts satisfied the partition ordering with respect to its neighboring parts.
\begin{itemize}
\item $a+\underbrace{c+c_r}_{\text{pair}}\xri{\text{forward adjustment}} \underbrace{(c+1)+c_r}_{\text{pair}}+(a-1).$

\item $a+\underbrace{c+(c-1)_r}_{\text{pair}}\xri{\text{forward adjustment}} \underbrace{c+c_r}_{\text{pair}}+(a-1).$
\end{itemize}
Let $j$ run from $1$ to $k$ sequentially such that we obtain $\la'$. Finally, color all uncolored parts blue in $\la'$. In this way we get the resulting bicolored partition $\la$. It is easy to see that $\la\in \sB(n, k)$ is well-defined and satisfied $|\al|+|\beta|=|\la|$ and $\ell(\beta)=\ell_r(\la)$.
\end{description}

On the other hand, for the inverse map $h^{-1}$, we can construct it step by step according to the map $h$. It should be noted that the ``forward adjustment'' becomes the ``backward adjustment'', that is:
\begin{itemize}
\item $\underbrace{c+c_r}_{\text{pair}}+a\xri{\text{backward adjustment}} (a+1)+\underbrace{c+(c-1)_r}_{\text{pair}}.$

\item $\underbrace{(c+1)+c_r}_{\text{pair}}+a\xri{\text{backward adjustment}} (a+1)+\underbrace{c+c_r}_{\text{pair}}.$
\end{itemize}
Therefore, the proof is complete.
\end{proof}

\begin{example}
For the bijection $h$ in Lemma \ref{lem:h}, we present a one-to-one correspondence example for $n=11$ as follows.
\\

\begin{minipage}[c]{0.5\textwidth}
\centering
\begin{tabular}{r|l}
$\sC$ & $\sB$\\
\hline
(11, $\ep$) & 11\\
(10+1, $\ep$) & 10+1\\
(9+2, $\ep$) & 9+2\\
(8+3, $\ep$) & 8+3\\
(7+4, $\ep$) & 7+4\\
(6+5, $\ep$) & 6+5\\
(8+2+1, $\ep$) & 8+2+1\\
(7+3+1, $\ep$) & 7+3+1\\
(6+4+1, $\ep$) & 6+4+1\\
(6+3+2, $\ep$) & 6+3+2\\
(5+4+2, $\ep$) & 5+4+2\\
(5+3+2+1, $\ep$) & 5+3+2+1\\
(9, 2) & 9+1+$1_r$\\
(8+1, 2) & 8+2+$1_r$\\
(7+2, 2) & 7+2+1+$1_r$\\
(6+3, 2) & 6+3+1+$1_r$\\
(5+4, 2) & 5+4+1+$1_r$\\
(6+2+1, 2) & 6+2+$2_r$+1\\
(5+3+1, 2) & 5+3+2+$1_r$
\end{tabular}
\end{minipage}
\begin{minipage}[c]{0.5\textwidth}
\centering
\begin{tabular}{r|l}
$\sC$ & $\sB$\\
\hline
(4+3+2, 2) & 4+3+2+1+$1_r$\\
(7, 4) & 7+2+$2_r$\\
(6+1, 4) & 6+3+$2_r$\\
(5+2, 4) & 5+3+$2_r$+1\\
(4+3, 4) & 4+3+2+$2_r$\\
(4+2+1, 4) & 4+3+$3_r$+1\\
(5, 6) & 5+3+$3_r$\\
(4+1, 6) & 4+$4_r$+3\\
(3+2, 6) & 4+$4_r$+2+1\\
(3, 8) & 5+$4_r$+2\\
(2+1, 8) & 5+$5_r$+1\\
(1, 10) & 6+$5_r$\\
(5, 4+2) & 5+2+$2_r$+1+$1_r$\\
(4+1, 4+2) & 4+3+$2_r$+1+$1_r$\\
(3+2, 4+2) & 3+$3_r$+2+$2_r$+1\\
(3, 6+2) & 4+$3_r$+2+1+$1_r$\\
(2+1, 6+2) & 4+$4_r$+2+$1_r$\\
(1, 8+2) & 5+$4_r$+1+$1_r$\\
(1, 6+4) & 4+$3_r$+2+$2_r$\\
\end{tabular}
\end{minipage}

\end{example}


\section*{Acknowledgement}
The author thanks the anonymous referee for his/her valuable comments and suggestions which have greatly improved the presentation of this paper.


\end{document}